\documentclass[a4paper,UKenglish,cleveref, autoref, thm-restate]{lipics-v2021}
\pdfoutput=1 %uncomment to ensure pdflatex processing (mandatatory e.g. to submit to arXiv)
\hideLIPIcs  %uncomment to remove references to LIPIcs series (logo, DOI, ...), e.g. when preparing a pre-final version to be uploaded to arXiv or another public repository

\title{Phase transition for tree-rooted maps} %TODO Please add

\author{Marie Albenque}{IRIF, Université Paris Cité, France}{malbenque@irif.fr}{}{}

\author{Éric Fusy}{Univ Gustave Eiffel, CNRS, LIGM, F-77454 Marne-la-Vallée, France}{eric.fusy@univ-eiffel.fr}{}{}

\author{Zéphyr Salvy}{Univ Gustave Eiffel, CNRS, LIGM, F-77454 Marne-la-Vallée, France}{zephyr.salvy@univ-eiffel.fr}{}{}

\authorrunning{M. Albenque, É. Fusy and Z. Salvy} %TODO mandatory. First: Use abbreviated first/middle names. Second (only in severe cases): Use first author plus 'et al.'

\Copyright{Marie Albenque, Éric Fusy and Zéphyr Salvy} %TODO mandatory, please use full first names. LIPIcs license is "CC-BY";  http://creativecommons.org/licenses/by/3.0/

\ccsdesc[500]{Mathematics of computing~Generating functions}
\ccsdesc[300]{Mathematics of computing~Enumeration}
\ccsdesc[300]{Mathematics of computing~Stochastic processes}
\ccsdesc[500]{Mathematics of computing~Probability and statistics}

\keywords{Asymptotic Enumeration, Planar maps, Random trees, Phase transition} %TODO mandatory; please add comma-separated list of keywords

\category{} %optional, e.g. invited paper

\relatedversion{} %optional, e.g. full version hosted on arXiv, HAL, or other respository/website
\acknowledgements{We would like to thank the anonymous referees for their careful reading and their insightful comments.}%optional

\nolinenumbers %uncomment to disable line numbering

\EventEditors{C\'{e}cile Mailler and Sebastian Wild}
\EventNoEds{2}
\EventLongTitle{35th International Conference on Probabilistic, Combinatorial and Asymptotic Methods for the Analysis of Algorithms (AofA 2024)}
\EventShortTitle{AofA 2024}
\EventAcronym{AofA}
\EventYear{2024}
\EventDate{June 17--21, 2024}
\EventLocation{University of Bath, UK}
\EventLogo{}
\SeriesVolume{302}
\ArticleNo{30}
\crefname{equation}{}{Equations} % Equations is always written with a capital
\Crefname{equation}{Equation}{Equations}

\renewcommand{\(}{\left(}
\renewcommand{\)}{\right)}

\newcommand{\N}{\mathbb{Z}_{\geq0}}
\newcommand{\C}{\mathbb{C}}

\newcommand{\m}{\mathfrak{m}}
\renewcommand{\b}{\mathfrak{b}}
\newcommand{\co}{\mathfrak{c}}
\renewcommand{\t}{\mathfrak{t}}
\newcommand{\M}{\mathrm{M}}
\newcommand{\B}{\mathrm{B}}

\newcommand{\LB}{\mathrm{LB}}

\newcommand{\cM}{\mathcal{M}}
\newcommand{\cB}{\mathcal{B}}

\newcommand{\E}[1]{\mathbb{E}\left[ #1 \right]}
\newcommand{\pr}[2][]{\mathbb{P}#1\( #2\)}
\newcommand{\prnu}[1]{\pr[_{n}^{(u)}]{#1}}

\def\bA{\overline{A}}
\def\bS{\overline{S}}
\def\bP{\overline{P}}

\def\Hu{H^{(u)}}
\def\Bu{B^{(u)}}
\def\Psiu{\Psi^{(u)}}

\def\Bn{\mathrm{B}_n}
\def\Mn{\mathrm{M}_n}
\def\Mnu{\mathrm{M}_n^{(u)}}
\def\Tnu{\mathrm{T}_n^{(u)}}

\def\diam{\mathrm{Diam}}

\newcommand{\tree}{\mathrm{T}} %arbre des blocs
\newcommand{\STnu}{\tau(\Mnu)} %arbre couvrant
\newcommand{\Pnu}{\mathbb{P}_n^{(u)}} %arbre couvrant

\def\Cat{\mathrm{Cat}}

\reversemarginpar
\begin{document}

\maketitle

% !TEX root = main.tex

\begin{abstract}
We introduce a model of tree-rooted planar maps weighted by their number of $2$-connected blocks. We study its enumerative properties and prove that it undergoes a phase transition. We give the distribution of the size of the largest $2$-connected blocks in the three regimes (subcritical, critical and supercritical) and further establish that the scaling limit is the Brownian Continuum Random Tree in the critical and supercritical regimes, with respective rescalings $\sqrt{n/\log(n)}$ and $\sqrt{n}$.
\end{abstract}

\section{Introduction}

A \emph{planar map} $\m$ is the proper embedding into the two-dimensional sphere of a connected planar finite multigraph, considered up to homeomorphisms. In recent years, models of random planar maps with weighted $2$-connected blocks~\cite{bonzom2016large,StuflerSurvey} have been introduced. In particular, the model with a Boltzmann weight $u$ per block exhibits a phase transition at $u_C=9/5$, with a ``tree phase'' for $u>u_C$ with only small blocks and having the Brownian Continuum Random Tree (CRT) as scaling limit~\cite{StuflerSurvey}, a ``map-phase'' for $u<u_C$ characterized by the existence of a giant block and having the Brownian sphere as scaling limit, and with mesoscopic blocks and the stable tree of parameter $3/2$ as scaling limit at the critical point $u_C$~\cite{fleurat2023phase}. 

Here, we study such a model in the context of decorated maps and consider the emblematic case of \emph{tree-rooted maps}, \emph{i.e}., maps endowed with a spanning tree. In theoretical physics, decorated maps are instrumental to provide models of two-dimensional quantum gravity coupled with matter. They lead to new asymptotic behaviours, and the study of scaling limits in that context is currently a very challenging topic in random maps~\cite{gwynne2020mating}. Among decorated maps, tree-rooted maps have very rich combinatorial properties and their enumeration goes back to Mullin~\cite{Mullin67}, who obtained the formula
\begin{equation}
m_n=\sum_{k=0}^{n}\binom{2n}{2k}\Cat_k\Cat_{n-k}=\Cat_n\Cat_{n+1}
\end{equation}
for the number of tree-rooted maps with $n$ edges, 
by observing that a tree-rooted map is a shuffle of two plane trees (the spanning tree and its dual). A direct bijective proof that $m_n=\Cat_n\Cat_{n+1}$ was later obtained by Bernardi~\cite{bernardi06}, who subsequently extended his bijection to maps endowed with a Potts model~\cite{bernardi08}.

Our contributions here are both enumerative and probabilistic. First, we show an asymptotic estimate for $2$-connected tree-rooted maps, and that the enumeration of tree-rooted maps with a weight $u$ per $2$-connected block undergoes a phase transition at an explicit (transcendental) value $u_C$. On the probabilistic side, we obtain limit laws for the sizes of the largest blocks, with the existence of a giant block if and only if $u<u_C$. Furthermore, we show that the scaling limit is the CRT for all $u\geq u_C$, with a discontinuity at $u_C$ for the order of magnitude of the rescaling, which is $\sqrt{n/\log n}$ at $u_C$ whereas it is $\sqrt{n}$ for $u>u_C$ (the scaling limit result for $u>u_C$ also follows from~\cite[Th 6.63]{StuflerSurvey}, as further commented in Section~\ref{sec:scalingLimit}). 
Finally, we discuss possible extensions in the concluding section.

\section{Block-decomposition of tree-rooted maps}
Let $\m$ be a planar map. We denote, respectively, by $E(\m)$, $V(\m)$ and $F(\m)$ its sets of edges, vertices and faces.
Any edge is made of two half-edges (which meet at the middle of the edge).
All the maps considered in this paper will be \emph{rooted}, meaning that one of their half-edges is distinguished (and is represented by an oriented edge on figures), and a rooted planar map will be simply called a \emph{map} from now on.
The \emph{size} of a map $\m$~--~denoted by $|\m|$~--~is defined as its number of edges.

A map $\m$ is said to be \emph{separable} if $E(\m)$ can be partitioned into two non-empty subsets $E_1$ and $E_2$ such that there exists exactly one vertex~--~called \emph{cut vertex}~--~incident to both an element of $E_1$ and an element of $E_2$. It is is said to be \emph{$2$-connected} otherwise. By convention, the vertex map (\emph{i.e.}, the map reduced to a single vertex) is considered to be $2$-connected.
A \emph{block} of $\m$ is a maximal $2$-connected submap \emph{of positive size}. The number of blocks of $\m$ is denoted by $b(\m)$, so that if $|\m| > 0$, then $b(\m)=1$ if and only if $\m$ is $2$-connected.

Fix $\m$ a map, and let $\tau$ be one of its spanning trees, then one calls $(\m,\tau)$ a tree-rooted map. We denote, respectively, by $\cM$ and $\cB$ the set of tree-rooted maps and of $2$-connected tree-rooted maps, by $\cM_n$ and $\cB_n$ the subset of $\cM$ and $\cB$ restricted to elements of size $n$, and by $M(z)$ and $B(y)$ the associated generating series. In the following, a tree-rooted map will be denoted by $\m$ instead of being explicitly written as a pair, and, for $\m\in \cM$, we write $\tau(\m)$ for its distinguished spanning tree.
\medskip

To enumerate $2$-connected maps, Tutte~\cite{tutte_1963} formulated a decomposition to relate the generating series of maps and of $2$-connected maps, as follows.
Fix $\m$ a map and let $\b$ be the block containing its root. For each half-edge $e$ of $\b$ incident to a vertex $u$, let $c$ be the corner of $\b$ incident to $u$ and following $e$ in counterclockwise order around $u$. The \emph{pendant submap} $\m_{e}$ of $e$ is defined as the maximal submap of $\m$ disjoint from $\b$ except at $u$, and located in the area of $c$. Unless $\m_e$ is reduced to the vertex map, its root is defined at the  half-edge following $e$ in counterclockwise order around $u$ in $\m$. From $\b$ and the collection of pendant submaps $\{\m_{e}\}$, we can bijectively reconstruct $\m$.

This decomposition extends readily to tree-rooted maps as follows. Fix $\m\in \cM$. Consider $(\b,\tau_\b)$, where $\b$ is, as before, the block of $\m$ containing its root and $\tau_\b=\tau(\m)\cap \b$. We claim that $\tau_\b$ is a spanning tree of $\b$: clearly, $\tau_\b$ is acyclic since $\tau(\m)$ is. Then, for any $u,v\in \b$, since $\b$ is $2$-connected, any path between $u$ and $v$ that is not included in $\b$ has to visit the same cut vertex at least twice, and in particular is not simple. Any simple path between $u$ and $v$ in $\m$ is then included in $\b$, and so is the unique simple path between $u$ and $v$ in $\tau(\m)$. This proves that $\tau_\b$ is connected. The same reasoning can be applied to all the pendant submaps $\m_e$, to get a similar decomposition in the tree-rooted case, which induces the following identity of generating series.
\begin{proposition}
The generating series satisfy the following equality:
\begin{equation}\label{eqn-cartes-boisees}
M(z) = B\(zM(z)^2\).
\end{equation}
Moreover, this equation can be refined to account for the number of blocks in a tree-rooted map. Writing $M(z,u) = \sum_{\m\in \cM}z^{|\m|}u^{b(\m)}$, one has:
\begin{equation}\label{eq:MB}
M(z,u)=uB(zM(z,u)^2) +1-u.
\end{equation}
\end{proposition}
 Note that these relations are exactly the same as the ones obtained in the non--tree-rooted case~\cite{tutte_1963,fleurat2023phase}.

\begin{figure}
\centering
\includegraphics[height=0.2\textheight]{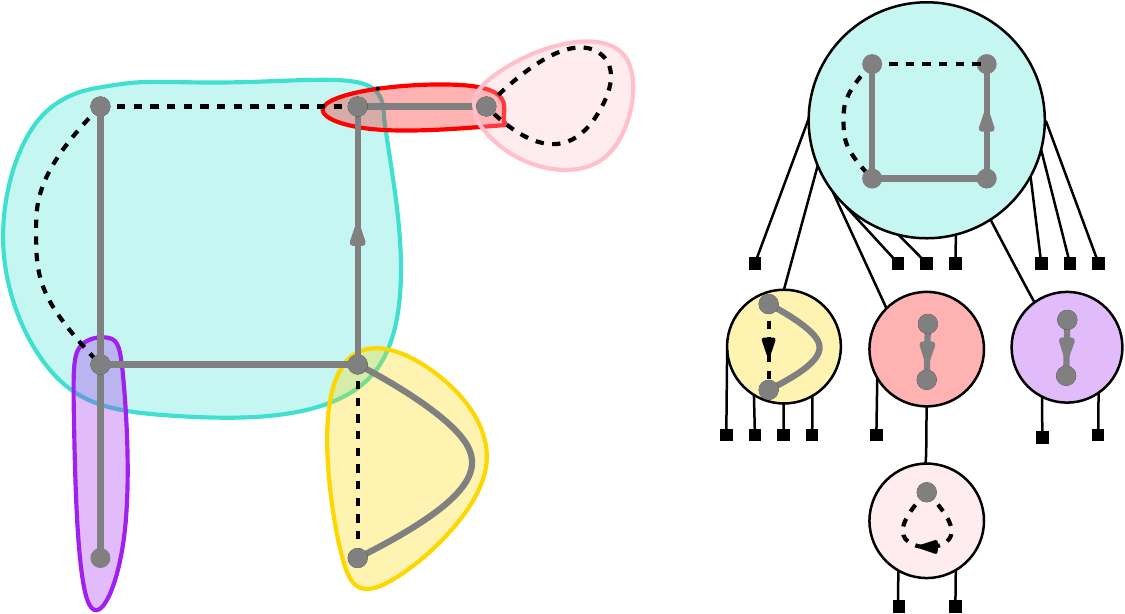}
\caption{Block tree corresponding to a tree-rooted planar map. Full grey (resp. dashed black) edges represent edges that are part (resp. not part) of the decorating spanning tree.}
\label{fig:block-tree-boisees}
\end{figure}

Tutte's decomposition can also be applied recursively, by considering first the root block and then applying the block decomposition to each of the pendant submaps. This can be encoded by a decomposition tree $T_{\m}$, which was first explicitly described by Addario-Berry in the non--tree-rooted case \cite[\S2]{2Louigi}, but which can also be extended to the tree-rooted case; see \cref{fig:block-tree-boisees}.
\begin{proposition}
\label{prop-t-m-n-boisee}
The block tree $T_\m$ of a tree-rooted map $\m$ satisfies the following properties:
\begin{itemize}
\item Edges of $T_\m$ correspond to half-edges of $\m$;
\item Internal nodes of $T_\m$ correspond to blocks of $\m$: if an internal node $v$ of $T_\m$ has $r$ children, then the corresponding block $\b_v$ of $\m$ has size $r/2$;
\item The map $\m$ is entirely determined by $\Big(T_\m, (\b_v, v\in T_\m)\Big)$ where $\b_v$ is the block of $\m$ represented by $v$ in $T_\m$ if $v$ is an internal node and is the vertex map otherwise.
 \end{itemize}
\end{proposition}

\section{Asymptotic enumeration}
\subsection{Asymptotic enumeration of \texorpdfstring{$2$}{2}-connected tree-rooted maps}

We obtain here an asymptotic estimate for the number $b_n:=[y^n]B(y)$ of $2$-connected tree-rooted maps of size $n$. The steps are as follows: we first lift (\cref{lem:sing_M}) the asymptotic estimate $m_n\sim\frac{4}{\pi n^3}16^n$ for tree-rooted maps to a singular expansion for the generating function $M(z)$. Then via~\cref{eqn-cartes-boisees}, we get in \cref{prop:dvp-B-boisee} the radius of convergence $\rho_B$ and the singular expansion of $B(y)$ around $\rho_B$. In order to transfer the singular expansion to an asymptotic estimate for $b_n$, we also show that $\rho_B$ is the unique dominant singularity of $B(y)$, using a combinatorial argument.

\begin{lemma}\label{lem:sing_M}
When $z\to \rho_M = \frac{1}{16}$ in $\C \setminus{} \{z\geq \rho_M\}$, one has, with $Z=1-16z$,
\begin{equation}
\label{dvp-M-cartes-boisees}
M(z) = 8 - \frac{64}{3\pi} - 8 \(\frac{10}{3\pi} - 1\) Z - \frac{2}{\pi}\ln\(Z\) Z^2+ O\(Z^2\).
\end{equation}
\end{lemma}

\begin{proof}
The explicit expression of $m_n$ translates to a $D$-finite equation satisfied by $M(z)$:
\begin{equation}
\label{diffeq-homo-cartes-boisees}
M'''(z) + \frac6z M''(z) + \frac{6(18z - 1)}{z^2(16z - 1)} M'(z) + \frac{12}{z^2(16z - 1)}M(z) = 0.
\end{equation}
$D$-finite equation theory~\cite[Sec. VII.9.1, p. 518]{flajolet2009analytic} gives that the finite singularities of a solution $f(z)$ of \cref{diffeq-homo-cartes-boisees} are among the zeroes of the denominators of the coefficients: $\mathcal{S} = \{0, 1/16\}$; and any solution of \cref{diffeq-homo-cartes-boisees} is analytically continuable along any path avoiding $\mathcal{S}$. In particular, the solution $M(z) = \sum_{n\geq 0} \mathrm{Cat}_n \mathrm{Cat}_{n+1} z^n$, which is clearly analytic at $0$, is continuable to the whole complex plane slit by the half-line $z \geq 1/16$.

Moreover, $1/16$ is a so-called \emph{regular} singularity;  
and, using the \verb|DEtools| package of the \verb|Maple| computer algebra software, one can compute singular expansions for a basis of solutions of~\cref{diffeq-homo-cartes-boisees}. The singular expansion of $M(z)$
is then a linear combination of the basis' singular expansions, which gives:
\[M(z) = \sum_{k=0}^\infty a_k Z^k - \ln(Z) \sum_{k=2}^\infty b_k Z^k,\quad\text{with}\quad Z=1 - 16z,\]
holding in a slit neighborhood of $1/16$. From the explicit expression $[z^n]M(z)=\Cat_n\Cat_{n+1}$ it follows that $a_0 = M(1/16) = 8 - \frac{64}{3\pi}$ and $a_1 = M'(1/16)/(-16) = -8\(\frac{10}{3\pi} - 1\)$. By Pringsheim's theorem, $M(z)$ is singular at its radius of convergence $1/16$ so there exists a smallest integer $k \geq 2$ such that $b_k \ne 0$. By applying transfer theorems~\cite[Chap. VI]{flajolet2009analytic}, one has
\[[z^n] M(z) \sim (-1)^{k} \frac{b_k k!}{n^{k+1}} 16^n.\]
Since $\Cat_n\Cat_{n+1}\sim\frac{4}{\pi n^3}16^n$, one must have $k=2$ and $b_2 =\frac{2}{\pi}$, which concludes the proof.
\end{proof}

\begin{proposition}
\label{prop:dvp-B-boisee}
The radius of convergence of $B(y)$ is
\begin{equation}
\label{eq:rho-b}\rho_B := \rho_M M^2(\rho_M) = \frac{4(3\pi - 8)^2}{9\pi^2} \approx 0.091,
\end{equation}
and, when $y\to\rho_B$ in a $\Delta$-neighbourhood of $\rho_B$, one has, with $Y=1-y/\rho_B$,
\begin{equation}
B(y) = 8 -\frac{64}{3 \pi} - \frac{8 (10 -3 \pi) (3 \pi-8 )}{9 \pi (4-\pi)}\, Y -\frac{2(3\pi-8)^3}{27\pi(4 - \pi)^3 }\, \ln (Y) Y^{2}+O (Y^{2}).
\end{equation}
Moreover, $\rho_B$ is the unique dominant singularity of $B(y)$. 
\end{proposition}

\begin{remark}
The generating series $B(y)$ is not $D$-finite (having a transcendental radius of convergence), but from~\cref{diffeq-homo-cartes-boisees,eqn-cartes-boisees} it is $D$-algebraic.
\end{remark}

\begin{proof}[Proof of \cref{prop:dvp-B-boisee}]
Let $H(z)=zM(z)^2$, so that one has $M(z)=B(y)$, where $y=H(z)$. Note that $H(z)$ has radius of convergence $1/16$, and it inherits from $M(z)$ a singular expansion of the form (with $Z=1-16z$):
$$
H(z)=\tau-\kappa\, Z+\xi\,\ln(Z)Z^2+O(Z^2)
$$
for $\tau,\kappa,\xi$ explicit (\emph{e.g.} $\tau=(8-\frac{64}{3\pi})^2/16$).

The function $H$ is analytic on $D(0,1/16)$ and since $H'(z)>0$ for any $z\in [0,1/16)$, one can apply the analytic local inversion theorem at any such value of $z$. Moreover, $H$ maps the interval $[0,1/16]$ to the interval $[0,\tau]$, so one can define a functional inverse $g$ of $H$ on a neighborhood of $[0,\tau)$, which is analytic on this domain.
Furthermore, $H'$ is continuous in a $\Delta$-neighbourhood of $1/16$, with positive value at $1/16$, hence $H(z)$ is injective on a $\Delta$-neighbourhood $U$ of $1/16$, and maps $U$ to an open region containing a $\Delta$-neighbourhood $V$ of $\tau$.

Using bootstrapping, the singular expansion of $g(y)$ at $\tau$ (valid in $V$) is easily obtained from the singular expansion of $H(z)$. With $Y=1-y/\tau$, one gets
$$
z=g(y)=\frac{1}{16} -\frac{\tau}{16\kappa}\,Y-\frac{\tau^2}{16\kappa^3}\,\ln(Y)Y^2+O(Y^2).
$$
With $Z=1-16z$, this gives
$$
Z=\frac{\tau}{\kappa}\,Y +\frac{\tau^2}{\kappa^3}\,\ln(Y)Y^2+O(Y^2).
$$
Then, $B(y)=M(g(y))$ is analytic at every point in $[0,\tau)$, and
the claimed singular expansion of $B(y)$ at $\rho_B:=\tau$ 
is obtained by composing the singular expansion of $M(z)$ with the singular expansion of $g(y)$, \emph{i.e.}, injecting the above expansion of $Z$ into the expansion in \cref{lem:sing_M}.
By Pringsheim's theorem, $\rho_B$ is the radius of convergence of $B(y)$.

It remains to prove that $\rho_B$ is the unique dominant singularity of $B(y)$. To do so, we use the trick of writing $B(y)$ as a supercritical composition scheme (in the sense of~\cite{Gourdon98}), which we achieve thanks to a decomposition into series-parallel components. Doing so, we prove in \cref{lem:seriesParallel} that $B(y)$ can be written as
\begin{equation}
\label{eq:BQA}
B(y)=1+ 2y + 2yA(y)+ yA'(y)Q(A(y)),
\end{equation} 
for some generating functions $A(y)$ and $Q(w)$ with nonnegative coefficients, such that $A(y)$ is non-periodic and has radius of convergence larger than $\rho_B$. This implies that the radius of convergence of $Q(w)$ is $A(\rho_B)$. Moreover, by the Daffodil Lemma, see~\cite[Lem. IV.1, p.~266]{flajolet2009analytic}, for any $y\neq \rho_B$ such that $|y|=\rho_B$, we have that $|A(y)|<A(\rho_B)$. Hence, $A(y)$ belongs to the disk of convergence of $Q$ and $y$ cannot be a singularity, which concludes the proof.
\end{proof}

Let $Q(w)$ be the generating function of $2$-connected tree-rooted maps with no face of degree $2$ nor vertex of degree $2$, with $w$ counting the number of non-root edges. A $2$-connected map with at least 2 edges is called \emph{series-parallel} if it has no $K_4$ minor. Let $A(y)$ (resp. $\bA(y)$) be the generating function of $2$-connected tree-rooted series-parallel maps such that the root-edge is not (resp. is) in the spanning tree, the variable $y$ counting the number of non-root edges.

\begin{lemma}\label{lem:seriesParallel}
The generating series $B(y)$, $A(y)$ and $Q(w)$ satisfy the identity \cref{eq:BQA}.

Moreover, the radius $\rho_A$ of convergence of $A$ satisfies $\rho_A=2-3\cdot 2^{-2/3} \approx 0.11>\rho_B$.
\end{lemma}

\begin{proof}
A \emph{series-parallel network} $\mathrm{N}$ is obtained by deleting the root-edge $e$ of a series-parallel map, the two extremities of $e$ being called the \emph{poles} of $\mathrm{N}$ (which are distinguished as the \emph{source}, the origin of $e$, and the \emph{sink}, the end of $e$). Note that $A(y)$ is also the generating function of series-parallel networks endowed with a spanning tree, while $\bA(y)$ is the generating function of
series-parallel networks endowed with a spanning forest made of two trees containing each of the two poles. These two cases are respectively called \emph{crossing} and \emph{non-crossing}. 

The \emph{core} $\co$ of a $2$-connected tree-rooted map $\b$ of size $|\b| \geq 2$ is obtained by repeatedly collapsing faces of degree $2$ and erasing vertices of degree $2$ (turning the two incident edges into a single edge). This process is actually well-behaved only if $\b$ is not series-parallel (otherwise it ends at a loop-edge with no vertex). Conversely, a $2$-connected map $\b$ is obtained from its core $\co$ where every edge is replaced by a series-parallel network. By convention, the root-edge of $\co$ is chosen as the one bearing the series-parallel network containing the root-edge of $\b$. If $\b$ is endowed with a spanning tree $\tau := \tau(\b)$, then for each edge $e$ of $\co$, letting $\mathrm{N}_e$ be the associated series-parallel network, on $\mathrm{N}_e$ the tree $\tau$ induces either a spanning tree, or a spanning forest with two trees containing each of the two poles. In the first case, $e$ is declared a tree-edge of $\co$, and thus $\tau$ induces a spanning tree on $\co$. In terms of generating functions, $Q(w)$ is the counting series for the core, each non-root edge of the core then contributing either $A(y)$ if a tree-edge (crossing case) or contributing $\bA(y)$ otherwise (non-crossing case). Since, by duality we have $A(y)=\bA(y)$, every non-root edge turns out to have the same contribution $A(y)$. On the other hand, the root-edge of the core contributes $A'(y)$ because of the choice of the root-edge. This yields the claimed equation~\cref{eq:BQA}, where the added term $2yA(y)$ accounts for $2$-connected tree-rooted series-parallel maps.

Now, to get the statement about $\rho_A$, it is well-known that a series-parallel network is either reduced to a single edge, or made of at least two series-parallel networks connected in series, or made of at least two series-parallel networks connected in parallel.
The series-parallel decomposition then yields the following equation-system:
\begin{align*}
A(y)&=y+S(y)+P(y),& \bA(y)&=y+\bS(y)+\bP(y),\\
S(y)&=\frac{(y+P(y))^2}{1-y-P(y)},& \bS(y)&=(y+\bP(y))\( \frac1{(1-y-P(y))^2}-1 \),\\
P(y)&=(y+S(y))\( \frac1{(1-y-\bS(y))^2}-1 \),& \bP(y)&= \frac{(y+\bS(y))^2}{1-y-\bS(y)}.
\end{align*}
By symmetry, one has $S(y)=\bP(y)$ and $P(y)=\bS(y)$, which yields $A(y)=\bA(y)$ (this is also clear by duality).
Hence, the function $A(y)$ is algebraic and satisfies\footnote{The series $A(y)$ is represented in  Sloane's OEIS by the sequence A121873, which enumerates non-crossing plants in the $(n+1)$-sided regular polygon \cite{Chap07}.}
\[A \! \left(y \right)^{3}+\left(y +1\right) A \! \left(y \right)^{2}+\left(2 y -1\right) A \! \left(y \right)+y = 0\]
and its radius of convergence is the smallest positive root of $4y^3 - 24y^2 + 48y - 5$ which gives $\rho_A = 2 - 3 \cdot 2^{-2/3} \approx 0.11$ and is larger than $\rho_B$.

\end{proof}

From \cref{prop:dvp-B-boisee}, by applying transfer theorems~\cite[Chap.VI]{flajolet2009analytic} one directly obtains:
\begin{corollary}
\label{asymptB-boisee}
When $n\to\infty$,
\begin{equation}\label{eq:bn}
b_n\sim \frac{4 \(3 \pi -8\)^{3}}{27 \pi \(4 - \pi \)^{3}} \cdot \rho_B^{-n} \cdot n^{-3}.
\end{equation}
\end{corollary}

\subsection{Enumerative phase transition for block-weighted tree-rooted maps}
In this section, we investigate the singular expansion of $z\mapsto M(z,u)$ around its radius of convergence. We prove that this expansion exhibits three possible behaviours depending on the value of $u$.

\begin{proposition}[Definitions of $u_C$ and of $y(u)$]\label{prop:defUc}
Recall that $\rho_B$ is the radius of convergence of $B(y)$. For $u\geq 0$, the equation
\begin{equation}\label{eq:yu}
\frac{2yuB'(y)}{uB(y)+1-u}=1
\end{equation}
has a unique solution in $[0,\rho_B]$, denoted by $y(u)$, if and only if $u\geq u_C$, where
\begin{equation}\label{eq:uC}
u_C:=\frac{9 \pi \(4-\pi \)}{ 420 \pi - 81 \pi^{2}- 512}\simeq 3.02.
\end{equation} 
Moreover, we set $y(u):=\rho_B$ for $u\leq u_C$.
\end{proposition}

\begin{remark}
The function $u\mapsto y(u)$ is non-increasing. It is plotted in \cref{fig:y-u}.
\end{remark}

The value of $u_C$ defined above is the \emph{critical point} of the model, and permits to identify three regimes for which the singular behavior of $M(z,u)$ differs:
\begin{proposition}\label{prop:Mu}
For $u>0$, let $\rho(u)$ be the radius of convergence of $z\mapsto M(z,u)$. Then, one has
\begin{equation}\label{eq:defRhou}
\rho(u)=\frac{y(u)}{(uB(y(u))+1-u)^2},
\end{equation}
and the following singular expansions hold in a $\Delta$-neighbourhood of $\rho(u)$, with $Z=1-z/\rho(u)$. 
\begin{itemize}
	\item
When $u<u_C$ (subcritical case),
\begin{equation}\label{eq:Musub}
M(z,u)=q(u)-r(u)\,Z-s(u)\ln(Z)Z^2+O(Z^2),
\end{equation}
where
$$
q(u)=1+u \(7-\frac{64}{3 \pi}\),\ \ r(u)=\frac{8 u \left(3 \pi -8\right) \left(10-3 \pi\right) \left(21 \pi u +3 \pi -64 u \right)}{\left(243 u -27\right) \pi^{3} - \left(1260 u -108\right) \pi^{2}+1536 \pi u},
$$
$$
s(u)=\frac{2 u \left(21 \pi u +3 \pi -64 u \right)^{3} \left(3 \pi -8\right)^{3}}{\pi \left(81 \pi^{2} u + 512 u +36 \pi -420 \pi u -9 \pi^{2} \right)^{3}}.
$$
\item
When $u=u_c$ (critical case), 
\begin{equation}\label{eq:Mucrit}
M(z,u)=q_C - s_C\, \ln(1/Z)^{-1/2}Z^{1/2}+O(Z),
\end{equation}
where
$$
q_C=q(u_C)=\frac{864 \pi - 144 \pi^{2} - 1280}{420 \pi - 81 \pi^{2} -512},\ \ \ s_C=\frac{16 \sqrt{6}\, \left(10-3 \pi\right)^{\frac{3}{2}} \left(4-\pi \right)}{420 \pi - 81 \pi^{2}-512}.
$$
\item
When $u>u_C$ (supercritical case),
\begin{equation}\label{eq:Musuper}
M(z,u)=q(u)-s(u)\, Z^{1/2}+O(Z),
\end{equation}
where 
$$
q(u)=uB(y(u))+1-u,\ \ \ \ s(u)=\frac{u B(y(u)) + 1 - u}{\sqrt{1+2 y(u)\frac{B''(y(u))}{B'(y(u))}}}.
$$
\end{itemize}
Moreover, $\rho(u)$ is the unique dominant singularity of $z\mapsto M(z,u)$ for every $u>0$.
\end{proposition}
\begin{proof}
Let $\Hu(z):=zM(z,u)^2$, and $\Bu(y):=uB(y)+1-u$. 
Squaring the equation $M(z,u)=\Bu(\Hu(z))$ and multiplying both sides by $z$, one gets the functional equation 
$$
\Hu(z)=z\Bu(\Hu(z))^2,
$$
which is of Lagrangean type. The functional inverse of $\Hu(z)$ is thus $\Psiu(y):=y/\Bu(y)^2$, and the singular expansion of $\Hu(z)$ (and hence of $M(z,u)$) depends on whether $\frac{\mathrm{d}}{\mathrm{d}y}\Psiu(y)=0$~--~which is equivalent to~\cref{eq:yu}~--~admits a solution in $(0,\rho_B)$.

More precisely, for $u>u_c$, $\Hu(z)$ has a dominant singularity of square-root type at $\rho(u)=\Psiu(y(u))$ (see~\cite[Thm VI.6, p. 404-405]{flajolet2009analytic}), and the same holds for $M(z,u)=(\Hu(z)/z)^{1/2}$, with the constants in~\cref{eq:Musuper}.

In the limit case $u=u_C$, one has $\frac{\mathrm{d}}{\mathrm{d}y}\Psiu(y)=0$ at $y=\rho_B$, where one gets (with $Y=1-y/\rho_B$) the expansion $z=\Psiu(y)=\rho(u_C)+\xi\, Y^2\ln(Y)+O(Y^2)$ for some explicit $\xi>0$.
By inversion and bootstrapping, one gets $y=\Hu(z)=\rho_B-\rho_B\,\sqrt{2\rho(u_C)/\xi}\, \sqrt{Z/\ln(1/Z)}+O(Z)$, and a similar expansion holds for $M(z,u)=(\Hu(z)/z)^{1/2}$, with the explicit constants in~\cref{eq:Mucrit}.

For $u<u_C$, one has $\frac{2\rho_B uB'(\rho_B)}{uB(\rho_B)+1-u}<1$, and there is no solution to $\frac{\mathrm{d}}{\mathrm{d}y}\Psiu(y)=0$ on $[0,\rho_B]$.
At $\rho_B$, one gets the expansion $z=\Psiu(y)=\rho(u)-\kappa\,Y+\xi\, Y^2\ln(Y)+O(Y^2)$ for some explicit $\kappa,\xi>0$, and writing $Y=1-y/\rho_B$. By inversion and bootstrapping, one gets $y=\Hu(z)=\rho_B - \frac{\rho(u)\rho_B}{\kappa}\,Z - \frac{\rho(u)^2\rho_B}{\kappa^3}\, Z^2\ln(Z)+O(Z^2)$, and a similar expansion holds for $M(z,u)=(\Hu(z)/z)^{1/2}$, with the explicit constants in~\cref{eq:Musub}.

\medskip

Finally, for every fixed $u>0$, the equation for $\Hu(z)$ is of (non-periodic) Lagrangean type, hence $\rho(u)$ is the unique dominant singularity of $\Hu(z)$, and the same holds for $z\mapsto M(z,u)=\Bu(\Hu(z))$. 
\end{proof}

Applying transfer theorems to the expansions in \cref{prop:Mu} then gives:
\begin{corollary}
Let $u>0$. Then, with the notation of \cref{prop:Mu}, one has the following asymptotic estimates as $n\to\infty$.
\begin{itemize}
	\item 
When $u<u_C$,
\begin{equation}
[z^n]M(z,u)\sim 2s(u)\,\rho(u)^{-n}\, n^{-3}.
\end{equation}
\item
When $u=u_C$,
\begin{equation}
[z^n]M(z,u_C)\sim \frac{s_C}{2\sqrt{\pi}}\, \rho(u)^{-n}\, n^{-3/2}\ln(n)^{-1/2}.
\end{equation}
\item
When $u>u_C$,
\begin{equation}
[z^n]M(z,u)\sim \frac{s(u)}{2\sqrt{\pi}}\, \rho(u)^{-n}\, n^{-3/2}.
\end{equation}
\end{itemize}
\end{corollary}

\begin{remark}
Instead of analytic combinatorics methods, one can also obtain these results using probabilistic methods. In the subcritical case, they require having an estimate of the size of the largest block, which is provided in \cref{sec:size-blocks-boisees}.
\end{remark}

\section{Probabilistic study of tree-rooted maps}
The purpose of this section is to study the phase transition undergone by a random tree-rooted maps weighted by their number of $2$-connected blocks. Following~\cite{2Louigi,fleurat2023phase}, this study is based on a interpretation of the block tree as a Bienaymé--Galton--Watson process (\cref{sub:BGW}), to obtain asymptotic estimates on the size of the largest blocks (\cref{sec:size-blocks-boisees}), and scaling limit results in the critical and supercritical cases (\cref{sec:scalingLimit}).

\subsection{Definition of the probabilistic model and Bienaymé--Galton--Watson trees}\label{sub:BGW}
We consider the following probability distribution on the class $\mathcal{M}$ of tree-rooted maps, indexed by a parameter $u>0$: for any integer $n\geq 0$, we define
\begin{equation}\label{eq:defPrnu}
\prnu{\m} = \frac{u^{b(\m)}}{[z^n] M(z,u)}\quad\text{for any }\m \in \mathcal{M}_n.
\end{equation}
We denote by $\Mnu$ a tree-rooted map sampled from $\Pnu$, by $\Tnu$ the block tree associated to it, and by $(\B_v,v\in \Tnu)$ its corresponding decorations.

\medskip
For $n\in \mathbb{Z}_{\geq 0}$ and $\mu$ a probability distribution on $\N$, $GW(\mu, n)$ denotes the law of a Bienaymé--Galton--Watson tree with offspring distribution $\mu$ conditioned to have $n$ edges.
For $u>0$ and $y\in[0,\rho_B]$, let $\mu^{y,u}$ be the probability distribution on $\mathbb{Z}_{\geq 0}$ defined  by setting, for $j\geq 0$, 
\begin{equation}
\label{mu-boisee}
\mu^{y,u}(2j) := \frac{b_j y^j u^{{1}_{j\ne 0}}}{uB(y)+1-u},\qquad\text{ so that }\qquad
\E{\mu^{y,u}}=\frac{2uyB'(y)}{uB(y)+1-u}.
\end{equation}

Moreover, we set $\mu^{u} := \mu^{y(u),u}$, where $y(u)$ is defined in~\cref{prop:defUc}. Then, the following proposition is the tree-rooted analogue of \cite[Proposition 3.1]{fleurat2023phase} (itself an extension of \cite[Proposition 3.1]{2Louigi}).
\begin{proposition}
\label{th-trees-boisees}
For every $u>0$ and any $n\geq 0$, under $\Pnu$, the law of the tree of blocks $(\Tnu,(\B_v,v\in \Tnu))$ can be described as
follows.\begin{itemize}
\item $\Tnu$ follows the law $GW(\mu^u,2n)$;
\item Conditionally given $\Tnu=\t$,
the blocks $(\B_{v},v\in \t)$ are independent random variables, and, for $v\in\t$, $\B_v$ follows a uniform distribution on the set of blocks of size $k_v(\t)/2$, where $k_v(\t)$ is the number of children of $v$ in $\t$.
\end{itemize}
\end{proposition}

Therefore the behavior of $\Tnu$ will be driven by the properties of $\mu^u$. It follows from \cref{asymptB-boisee} and \cref{prop:defUc} that $\mu^u$ exhibits the following phase transition:

\begin{lemma}
\label{tree-critical-boisee}
For any $u>0$, define
\begin{equation}\label{eq:cu}
 c(u) = \frac{4 \(3 \pi - 8\)^3}{9 \(4 - \pi \)^3} \frac{u}{(21 \pi - 64) u + 3 \pi}.
\end{equation}
Then, one has:
\begin{description}
\item[Subcritical case.] For $u < u_C$,
\begin{equation}\label{Eu-boisee}
E(u) := \E{\mu^{u}} = \frac{16 \(3 \pi - 8\) \(10 - 3 \pi \)}{3 \(4 - \pi \)} \frac{u}{(21 \pi - 64) u + 3 \pi}<1\quad\text{and}\quad\mu^{u}(\{2j\}) \overunderset{}{j\to\infty}{\sim} c(u) j^{-3};
\end{equation}
\item[Critical case.] For $u = u_C$,
\[\E{\mu^u} =1\qquad\text{and}\qquad\mu^{u_C}(\{2j\}) \overunderset{}{j\to\infty}{\sim} c(u_C)j^{-3}=\frac{\(3 \pi -8\)^{2}}{12 \(10 - 3 \pi \) \(4 -\pi \)^{2}} j^{-3} \simeq 0.40\, j^{-3};\]
\item[Supercritical case.] For $u > u_C$,
\[\E{\mu^u} =1\qquad\text{and}\qquad\mu^{u}(\{2j\}) \sim c(u) \left(\frac{y(u)}{\rho_B}\right)^j j^{-3},\]
where $y(u)<\rho_B$ so $\mu^u$ has exponential moments.
\end{description}
\end{lemma}

\subsection{Phase transition for the sizes of the largest blocks.}
\label{sec:size-blocks-boisees}
This section puts into light a phase transition for the block sizes of random tree-rooted maps drawn according to $\Pnu$. We use probabilistic techniques to obtain the results, but an analysis using a saddle-point method  could also be carried out.

For $\m$ a tree-rooted map, denote by $\LB_1 (\m) \geq \dots \geq \LB_{b(\m)} (\m)$ the sizes of its blocks in decreasing order. By convention, we set $\LB_k(\m) = 0$ if $k > b(\m)$. For a random variable $X_n$ and a positive sequence $a_n$, recall that $X_n=O_\mathbb{P}(a_n)$ (resp. $X_n=\Theta_\mathbb{P}(a_n)$) means that $\(\mathbb{P}(X_n\leq a_n u_n)\)_{n\geq 0}$ (resp. $\(\mathbb{P}(a_n/u_n\leq X_n\leq a_nu_n)\)_{n\geq 0}$) tends to $1$ for any positive $u_n$ tending to $+\infty$.

\begin{theorem} The random tree-rooted map $\Mnu$, drawn according to $\Pnu$, exhibits the following behaviours when $n$ tends to infinity.
\begin{description}
    \item[Subcritical case.]
For $u<u_c$, the largest bloc is macroscopic, and more precisely one has:
\begin{equation}\label{eq:bloc1SsCritique}
\frac{\LB_{1}(\Mnu) - (1-E(u))n}{\sqrt{c(u) n \ln(n)}} \xrightarrow[n\to\infty]{(d)} \mathcal{N}(0,1).
\end{equation}
Furthermore, for any fixed $j\geq 2$, it holds that $\LB_j(\Mnu) = \Theta_{\mathbb{P}}(n^{1/2})$ and for $x>0$:
\begin{equation}\label{eq:bloc2SsCritique}
\pr{\LB_{j}(\Mnu) \leq x\sqrt{n}} \xrightarrow[n\to\infty]{} e^{-\lambda(x)} \sum_{p=0}^{j-2} \frac{\lambda(x)^p}{p!}, \qquad \text{ where }\lambda(x) := \frac{c(u)}{2x^2}.
\end{equation}

 \item[Critical case.] For $u = u_C$, for any fixed $j\geq 1$, it holds that $\LB_j(\Mnu) = \Theta_{\mathbb{P}}(n^{1/2})$. More precisely, up to a shift of indices, the sizes of the blocks exhibit a similar behavior as the sizes of non-macroscopic blocks in the subcritical regime, namely, for $x>0$:
 \begin{equation}\label{eq:blocCritique}
\pr{\LB_{j}(\Mnu) \leq x\sqrt{n}} \xrightarrow[n\to\infty]{} e^{-\lambda(x)} \sum_{p=0}^{j-1} \frac{\lambda(x)^p}{p!}, \qquad \text{ where }\lambda(x) := \frac{c(u_C)}{2x^2}.
\end{equation}
 \item[Supercritical case.] For $u > u_C$, for all fixed $j \geq 1$, it holds as $n\rightarrow\infty$ that
\[\LB_{j}(\Mnu) = \frac{\ln(n)}{\ln\(\frac{\rho_B}{y(u)}\)} - \frac{3\ln(\ln(n))}{\ln\(\frac{\rho_B}{y(u)}\)}+O_{\mathbb{P}}(1).\]
\end{description}
\end{theorem}

\begin{proof}
In all three cases, we make extensive use of Janson's survey \cite{JansonSurvey}. In the supercritical case, one can proceed as in the non--tree-rooted case \cite{fleurat2023phase} and use the survey's Theorem 19.16. In the critical case, its Example 19.29 can be applied. The subcritical case is a bit more involved, Janson's Theorem 19.34 can be applied for the size of the $j$-th largest block for $j\geq 2$. Results from Kortchemski \cite[Theorem~1]{kortchemskiLimitTheoremsConditioned2015} allow to conclude for the largest block.
\end{proof}

\begin{remark}
One can get a local limit theorem for $\LB_1(\Mnu)$ in the subcritical case as in~\cite{Stu20} (up to the technicality that nodes of the block tree have only even numbers of children). Furthermore, one can state a joint limit law for the sizes $\LB_j(\Mnu)$. For any fixed $r\geq 1$,
\[\(\frac{c(u)}{2} \(\frac{\LB_j(\Mnu)}{\sqrt{n}}\)^{-2},\ 2\leq j \leq r+1\) \xrightarrow[n\to\infty]{(d)} (A_1, \dots, A_{r}),\]
where the $A_i$ are the decreasingly ordered atoms of a Poisson Point Process of rate $1$ on $\mathbb{R}_+$. The same joint limit law holds at $u_C$ (with $j$ from $1$ to $r$). 
\end{remark}

\begin{remark}
In contrast to the case of non--tree-rooted maps \cite{fleurat2023phase}, here in the subcritical case $u<u_C$ the size of the second block is negligible compared to the order of fluctuation of the size of largest block.
Moreover, for $u<u_C$ and fixed $j\geq1$, the $j$-th largest critical block has the same limit law (up to constant rescaling) as the $j+1$-th largest block in the subcritical regime, which did not hold in the non--tree-rooted case.
 Informally, the conditioning that a random walk (subcritical case) is an excursion (critical case) has negligible effect on the law of the largest steps, so subcritical blocks (for $j\geq 2$) behave like critical blocks.
\end{remark}

\subsection{Scaling limit in the critical and supercritical cases}\label{sec:scalingLimit}
In the critical and the supercritical cases, we can establish the following convergence result: 
\begin{theorem}\label{thm:scalingLimit}
For any fixed $u\geq u_c$, there exist some constants $\alpha_u$, $\beta_u$ and $\gamma_u$ such that:
\begin{itemize}
\item If $u>u_c$, it holds that:
\begin{equation}
 \frac{\gamma_u}{\sqrt{n}}\cdot\left(\M_n^{(u)},\STnu,\tree_n^{(u)}\right)\xrightarrow[n\rightarrow\infty]{(d)}(\alpha_u\cdot \mathcal{T}_{\mathrm{e}},\beta_u \cdot \mathcal{T}_{\mathrm{e}},\mathcal{T}_{\mathrm{e}});
\end{equation}

\item If $u=u_c$, it holds that:
\begin{equation}
 \frac{\gamma_{u_C}\sqrt{\log(n)}}{\sqrt{n}}
 \left(\M_n^{(u_C)},\tau(\M_n^{(u_C)}),\tree_n^{(u_C)}\right)\xrightarrow[n\rightarrow\infty]{(d)}(\alpha_{u_C}\cdot\mathcal{T}_{\mathrm{e}},\beta_{u_C}\cdot\mathcal{T}_{\mathrm{e}},\mathcal{T}_{\mathrm{e}});
\end{equation}
\end{itemize}
where, in both cases, each $\mathcal{T}_{\mathrm{e}}$ is a copy of the same realization of the Brownian Continuum Random Tree (CRT), and the convergence holds in the Gromov-Hausdorff-Prokhorov sense.
\end{theorem}

Convergence towards the CRT in the supercritical case was previously obtained by Stufler~\cite[Theorem~6.63]{StuflerSurvey}, who considers a framework where a block-weighted map is sampled, and \emph{afterwards} one of its spanning trees is uniformly sampled. Block-weighted models of random tree-rooted maps fall in this model upon tweaking the weights \cite[Remark~6.65]{StuflerSurvey}. Our contribution lies in showing scaling limit in the critical case (and having a unified proof for both the supercritical and critical cases), and finding the value of $u_C$.

\begin{proof}[Proof of \cref{thm:scalingLimit}]
The convergence of the sequence of tree of blocks $(\tree_n^{(u)})$ follows from classical results about the scaling limit of Bienaymé--Galton--Watson trees towards the CRT~\cite{LeGallTrees,DuquesneLeGall}, with an offspring distribution with a finite second moment in the $u>u_c$ case or with finite moments of order $2+\varepsilon$ in the critical case. The values of the constant $\gamma_u$ follows from general results. In the case $u>u_c$, $\gamma_u=\sigma_u/2$, where $\sigma_u$ is the standard deviation of $\mu^u$, given explicitely by
\[
\sigma_u^2 = 1 + 2y(u) \frac{B''(y(u))}{B'(y(u))}.
\]
For $u=u_C$, the variance of $\mu^{u_C}$ is infinite, and we get $\gamma_{u_c}=\sqrt{2c(u_C)}$ (see \emph{e.g.}~\cite[Ex.~7.10]{JansonStable}).

To establish the scaling limit of $\Mnu$ and of $\STnu$, we proceed as in~\cite{StuflerSurvey,fleurat2023phase} and prove that the distances in the map and in the spanning tree are~--~up to a linear factor~--~equivalent in the limit to the distances in the block tree. The proof extends effortlessly, and only requires (for the critical case) to have a control on the diameter of the $2$-connected blocks and of their spanning trees, that we state in~\cref{lem:diameter}. The Prokhorov part of the convergence can be established along the same lines as~\cite[Lemma~5.13]{fleurat2023phase}.
\end{proof}

A sequence $(p_n)_{n\geq 0}$ of nonnegative real numbers is called \emph{stretched-exponential} if there is $\gamma>0$ such that $p_n\leq \exp(-n^{\gamma})$ for $n$ large enough.
\begin{lemma}[Bound on the diameter of random $2$-connected tree-rooted maps]\label{lem:diameter}
Let $B_n$ be a uniformly random $2$-connected tree-rooted map of size $n$, and recall that $\tau(B_n)$ denotes its distinguished spanning tree.
Then, for any $\varepsilon>0$, the sequences $\(\mathbb{P}(\diam(\Bn) \geq n^{1/2+\varepsilon})\)_{n\geq 0}$ and $\(\mathbb{P}(\diam(\tau(B_n)) \geq n^{1/2+\varepsilon}\)_{n\geq 0}$ are stretched-exponential\footnote{To prove \cref{thm:scalingLimit}, it is actually enough to have the result for an exponent strictly smaller than $1$.}.
\end{lemma}
\begin{proof}
Note first that it is enough to establish the result for $\diam(\tau(B_n))$ since $\diam(\Bn)\leq \diam(\tau(B_n))$ deterministically.

By Mullin's bijection~\cite{Mullin67}, for $\Mn$ a uniform element of $\cM_n$, the height of $\tau(\Mn)$ is distributed as the maximal abscissa $X_n$ in a random walk (with steps in $\{W,E,S,N\}$) of length $2n$, ending at the origin and staying in the right-hand upper quadrant.
It is easy to establish, \emph{e.g.} using Chernoff's bound and a union bound, that the maximal abscissa $\widetilde{X_n}$ in a random walk of length $2n$ in $\mathbb{Z}^2$ is such that, for any $\varepsilon>0$, the sequence $\Big(\mathbb{P}(\widetilde{X_n}\geq n^{1/2+\varepsilon})\Big)_{n\geq 0}$ is stretched-exponential.

Since the random walk has probability $\Theta(n^{-3})$ to end at the origin and to stay in the quadrant,
 the sequence $\(\mathbb{P}(X_n\geq n^{1/2+\varepsilon})\)_{n\geq 0}$ is also stretched-exponential. And so is the sequence $\(\mathbb{P}(\diam(\tau(\Mn))\geq n^{1/2+\varepsilon})\)_{n\geq 0}$ since the diameter of a tree is at most twice its height. Let $\alpha=1-E(1)$. By the results of the previous section, one gets
\[
\mathbb{P}(\LB_{1}(\M_{\lfloor n/\alpha\rfloor})=n)=\Theta(1/\sqrt{n\ln(n)}),
\]
and in that case a block of maximal size is distributed as $\Bn$. Hence, the sequence $\(\mathbb{P}(\diam(\tau(\Bn))\geq n^{1/2+\varepsilon})\)_{n\geq 0}$ is also stretched-exponential.
\end{proof}

\section{Perspectives}

It has been recently shown~\cite{ding2020fractal,gwynne2020mating,gwynne2019bounds} that for $\Mn$ a random tree-rooted map of size $n$, the volume-growth exponent (whose inverse should give the exponent for the order of magnitude of the diameter) is with high probability in the interval $[3.550408,3.63299]$. It would be interesting to verify whether these bounds also hold for the random $2$-connected tree-rooted map $\B_n$, and more generally for the random tree-rooted map $\Mnu$ in the subcritical regime.

Regarding extensions of the model, one could consider maps endowed with a spanning forest, with weight $v$ per tree in the forest, which were studied by Bousquet-Mélou and Courtiel~\cite{bousquet2015spanning}, and one could additionally have a weight $u>0$ per $2$-connected block. They showed that, for $v>0$, one gets the asymptotic behaviour $n^{-5/2}$ as in pure maps~\cite{bousquet2015spanning}. The phase transition should thus be of the same nature than for the non--tree-rooted case~\cite{fleurat2023phase}, and we expect the scaling limit to be a stable tree of parameter $3/2$ at the critical weight $u_C(v)$. The case $v=0$ corresponds to tree-rooted maps as studied here. Interestingly, their model still has a combinatorial interpretation for $v\in[-1,0)$, with asymptotic behaviour $n^{-3}\ln(n)^{-2}$~\cite{bousquet2015spanning}. From this behaviour it can be expected that, at the critical weight $u_C(v)$, the asymptotic enumeration has a correcting term $n^{-3/2}$ and the scaling limit is the CRT with distances rescaled by $n^{1/2}$ (same order of magnitude as in the supercritical case). To have a continuous range of asymptotic exponents, one could more generally consider random maps weighted by a Potts model, and additionally weighted at blocks (a method to derive the singular exponents of general maps weighted by a Potts model has  been developed in~\cite{borot2012loop}; see also~\cite{eynard1999potts,bernardi2011counting}).

Finally, one could also consider other kinds of block-decompositions in the context of decorated maps, such as $3$-oriented triangulations or $2$-oriented quadrangulations decomposed into irreducible components, having a weight $u$ per such component. The asymptotic exponents at $u=1$ are $n^{-5}$ and $n^{-4}$~\cite{felsner2010asymptotic}, respectively. This suggests that, as in the above mentioned model, at the critical weight $u_C$, the model exhibits a tree-behaviour: the asymptotic enumeration has polynomial correction $n^{-3/2}$ and rescaling the distances by $n^{1/2}$ gives convergence towards the CRT.

\bibliography{AofaBoisees}

\appendix

% !TEX root = main.tex

\begin{figure}[t]
\centering
\includegraphics[height=0.37\textheight]{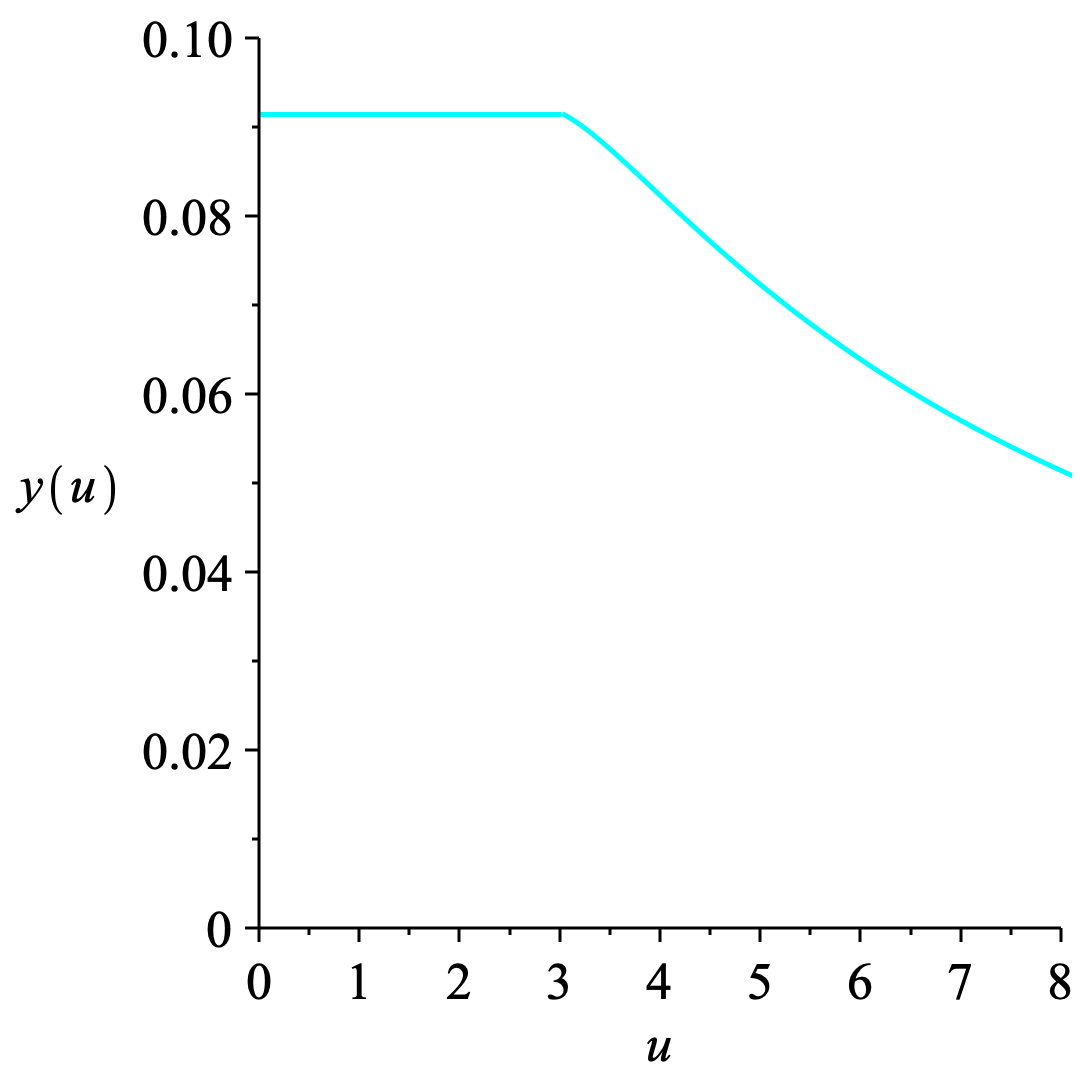}
\caption{Plot of $y(u)$~defined in \cref{prop:defUc}.
%~--~as function of $u$.
At $u_C$ the tangent from the right side is also horizontal.}
\label{fig:y-u}
\end{figure}

\end{document}